\documentclass{amsart}

\usepackage{amsmath}
\usepackage{amstext}
\usepackage{amssymb}

\usepackage{amsthm}

\theoremstyle{plain}
\newtheorem{theorem}{Theorem}[section]
\newtheorem{lemma}[theorem]{Lemma}

\newtheorem{thrm}{Theorem} 

\theoremstyle{definition}
\newtheorem{definition}[theorem]{Definition}%
\numberwithin{equation}{section}

\usepackage{cite}

\makeatletter
\newcommand{\iflabelexists}[3]{\@ifundefined{r@#1}{#3}{#2}}
\makeatother

\begin{document}

\title[]%
{Sums of Weighted Differentiation Composition Operators}
\author{Soumyadip Acharyya}
\address[Soumyadip Acharyya]{Department of Math, Physical and Life Sciences\\Embry-Riddle Aeronautical University Worldwide\\Daytona Beach, FL 32114-3900}
\email{acharyys@erau.edu}
\thanks{The first author acknowledges that this research was partially supported through Embry-Riddle Aeronautical University research funding, Agency Award ID: 13368.} 
\author{Timothy Ferguson}
\address[Timothy Ferguson]{Department of Mathematics\\University of Alabama\\Tuscaloosa, AL}
\email{tjferguson1@ua.edu}
\thanks{The second author was partially supported by RGC Grant RGC-2015-22 from the University of Alabama.}

\begin{abstract}
  We solve an interpolation problem in $A^p_\alpha$ involving specifying
  a set of (possibly not distinct) $n$ points, where the
  $k^{\textrm{th}}$ derivative at the   $k^{\textrm{th}}$ point is
  up to a constant as large as possible for functions of unit norm.
  The solution obtained has norm bounded by a constant independent of the
  points chosen. 
  As a direct application, we obtain a
  characterization of the order-boundedness of a sum of products of
  weighted composition and differentiation operators acting between
  weighted Bergman spaces. We also characterize the compactness of
  such operators that map a weighted Bergman space into the space of bounded
  analytic functions.
\end{abstract}

\maketitle

\section{Introduction}

Let $\varphi$ be an analytic self-map of the unit disk
$\mathbb{D}$ in the complex plane, and $H (\mathbb{D})$ be the
space of all holomorphic functions on $\mathbb{D}$. The
composition operator with symbol $\varphi$ is a linear operator
on $H(\mathbb{D})$ defined by
\begin{equation*}
  C_{\varphi}\left(f\right) = f \circ
   \varphi, \quad  f \in H \left(\mathbb{D}\right).
\end{equation*} 

An important generalization is the weighted composition operator
$uC_{\varphi}$ with weight $u$ (a measurable complex-valued
function on $\mathbb{D}$) and symbol $\varphi$, defined on
$H(\mathbb{D})$ as
\begin{equation*}
\left(uC_{\varphi}\right)\left(f\right) = u \cdot ( f\circ\varphi), \quad f \in H\left(\mathbb{D}\right).
\end{equation*}

The main purpose of studying (weighted) composition operators
acting on a holomorphic Banach space is to describe their
operator theoretic properties in terms of the function properties
of the associated symbols and weights. For basic theory of
composition operators and many earlier developments on a wide
variety of topics involving (weighted) composition operators, see
\cite{1,2}.

For a non-negative integer $n,$ let $D^n$ be the iterated
differentiation operator defined by
\begin{equation*}
  D^n \left(f\right) = f^{(n)}, \quad f \in H\left(\mathbb{D}\right).
\end{equation*}

During recent years, there has been a great interest in studying
product-type operators between holomorphic spaces. Let $\varphi$
be an analytic self-map of $\mathbb{D},$ and
$u \in H (\mathbb{D}).$ For a non-negative integer $n,$
$D^{n}_{u, \varphi}$ is the product-type operator defined as
\begin{equation*}
D^{n}_{u, \varphi}\, (f) = u \,.\, (f^{(n)} \circ \varphi) \,, \quad f \in H \left(\mathbb{D}\right).
\end{equation*}
Note that $D^{n}_{u, \varphi}$ can be written as
$D^{n}_{u, \varphi} = uC_{\varphi} \circ D^n$.

The operator $D^{n}_{u, \varphi}$ is called generalized
weighted composition operator, since it includes many known
operators. For example, if $n = 0$ and $u \equiv 1,$ we obtain
the composition operator with symbol $\varphi.$ If $n = 0,$ we
obtain the weighted composition operator $uC_{\varphi}$ with
weight $u$, and symbol $\varphi.$ When $u \equiv 1$ and
$\varphi (z) = z,$ $D^{n}_{u, \varphi}$ reduces to $D^{n}.$ If
$n = 1$ and $u(z) = \varphi^{'} (z),$ we obtain the product
$DC_{\varphi}.$ Thus generalized weighted composition operators
provide a unified way of
treating these operators. The operator $D^{n}_{u, \varphi}$ and
some of its special cases such as $DC_{\varphi}$ were studied in
\cite{3,4,5,6,7,8,9,10,11,12}. In particular, see
\cite{3,4,6,8,11,12,13,14}, which are focused on these types of
operators on weighted Bergman
spaces.
Some good references for weighted Bergman spaces are \cite{D_Ap, Zhu_Ap}. 

In \cite{14}, Colonna and Sharma studied the operator
$uC_{\varphi} + D^{n}_{v, \varphi}.$
In Theorem 2.7, they established the
following characterization of its order-boundedness between
weighted Bergman spaces. Let
$p, q \in (0, \infty), \alpha \in [-1, \infty), \beta \in (-1,
\infty)$ and let $n$ be a positive integer.
Then 
$uC_{\varphi} + D^{n}_{v, \varphi}: A^{p}_{\alpha} \rightarrow
A^{q}_{\beta}$ is order-bounded if and only if each summand is order
bounded, 
if and only if
\begin{align*}
  &\int_{\mathbb{D}} \dfrac{|u(z)|^q}{(1 - |\varphi(z)|^{2})^{(2+\alpha)\frac{q}{p}}}\,dA_{\beta}(z) < \infty, \\
  &\int_{\mathbb{D}} \dfrac{|v(z)|^q}{(1 - |\varphi(z)|^{2})^{(2+\alpha)\frac{q}{p} + nq}}\,dA_{\beta}(z) < \infty.
\end{align*}

Motivated by the line of study in \cite{14}, we investigate
the order-boundedness of the operator
\[
  T_{u;\varphi}^n = \displaystyle \sum_{k=0}^{n} D^{k}_{u_{k}, \varphi_{k}},
\]
where
$\varphi_{0}, \varphi_{1},.., \varphi_{n}$ are analytic self-maps
of $\mathbb{D},$ and $u_{0}, u_{1},.., u_{n} \in H(\mathbb{D}).$
We obtain the following characterization.

\newcommand{\OBExtensiontext}{
    Let
  $T^n_{u; \varphi} = \displaystyle \sum_{k=0}^{n} D^{k}_{u_{k},
    \varphi_{k}}$. Let
  $p, q \in (0, \infty), \alpha \in [-1, \infty), \beta \in (-1,
  \infty)$ and let $n$ be a positive integer. The following are
  equivalent.
\begin{enumerate} 
\renewcommand{\theenumi}{\alph{enumi}}
\item $T^n_{u; \varphi}: A^p_\alpha \rightarrow A^q_\beta$ is order bounded.
\item Each $D^{k}_{u_{k}, \varphi_{k}}$ is order bounded. 
\item For each $0 \leq k \leq n$ the symbols satisfy the condition
\[
 \int_{\mathbb{D}} \frac{|u_k(z)|^q}{(1 - |\varphi_k(z)|^2)^{(2+\alpha)\frac{q}{p} + kq}} dA_\beta (z) < \infty.
\]
\end{enumerate}
}

\begin{thrm}\label{OBExtension}
\OBExtensiontext
\end{thrm}

Taking $u_1 = u_2 = .. = u_{n-1} \equiv 0,$ and
$\varphi_n = \varphi_0$ in Theorem \ref{OBExtension}, we obtain
the above result in \cite{14}.

We also provide the following characterization of compactness of
$T^n_{u; \varphi}$ acting from a weighted Bergman space into
$H^{\infty}.$
\newcommand{\CompactExtensiontext}{
 Let $p \in (0, \infty)$ and $\alpha \in [-1, \infty)$ and $n$
  be a positive integer. Let $T = T^n_{u; \varphi}.$ The
  following are equivalent.
\begin{enumerate}
\renewcommand{\theenumi}{\alph{enumi}} 
\item
$T: A^p_\alpha \rightarrow H^\infty$ is compact.
\item For every bounded 
sequence $\{f_k\}$ in $A^p_\alpha$ converging to $0$ uniformly on compact subsets of $\mathbb{D}$, the sequence $\|Tf_k\|_\infty$ approaches $0$ as $k \rightarrow \infty$. 
\item
The weights $u_k \in H^\infty$ and for each $0 \leq k \leq n$ we have
\[
\lim_{|\varphi_k(z)| \rightarrow 1} 
   \frac{|u_k(z)|}{(1-|\varphi_k(z)|^2)^{\frac{2+\alpha}{p} + k}} = 0.
\]
\item Each summand of $T$ is compact.
 \end{enumerate}}
\begin{thrm}\label{CompactExtension}
 \CompactExtensiontext
\end{thrm}

Taking $u_0 = u, u_1 = u_2 = .. = u_{n-1} \equiv 0, u_n = v$ and $\varphi_n = \varphi_0 = \varphi$ in Theorem \ref{CompactExtension}, we obtain Theorem 3.1
from \cite{14}.

The most crucial tool used in the proof of Theorem \ref{OBExtension} and \ref{CompactExtension} is Theorem \ref{testfunction}, the following interpolation theorem in $A^p_\alpha$ which is discussed in section 3.  
\newcommand{\testfunctiontext}{
Let $0 < p \leq \infty$ and $-1 \leq \alpha < \infty$. 
  Let $\lambda_j \in \mathbb{D}$ for $0 \leq j \leq N$ where $N$ is a positive integer.  
  There is some constant $C_{N,p,\alpha}$ depending only on $N$, $p$, and
  $\alpha$ but not on the
$\lambda_j$ such that for each $J$ such that 
$0 \leq J \leq N$ there is a function 
$f$ such that $\|f\|_{p,\alpha} \leq C_{N,p,\alpha}$ and 
\[ 
f^{(J)}(\lambda_J) = \frac{1}{(1-|\lambda_J|^2)^{\frac{2+\alpha}{p} + J}} 
\]
and $f^{(k)}(\lambda_k) = 0$ for $k \neq J$.  
Moreover, as $\min\{|\lambda_j|: 0 \leq j \leq N\} \rightarrow 1$, the function 
$f(z)$ approaches $0$ uniformly when $z$ is restricted to a compact subset of 
$\mathbb{D}$. 
}
\begin{thrm}\label{testfunction}
\testfunctiontext
\end{thrm}

Preliminaries are gathered in the next section. The main results are proved in sections 3 and 4. Throughout the paper, $C$ stands for a positive constant depending on parameters that may change from one line to another. We say two positive functions $f$ and $g$ are equivalent, denoted by $ f \asymp g$, if there exist positive constants $C$ and $\tilde{C}$ independent of the associated variable(s) such that $C f \leq g \leq \tilde{C}f.$ 

\section{Preliminaries}
For $\alpha > -1$ and $0<p<\infty$, the weighted Bergman space $A_{\alpha}^p$ consists of all holomorphic functions $f$ defined on $\mathbb{D}$ such that \begin{equation*}
\|f\|_{p,\alpha}^p = \displaystyle \int_{\mathbb{D}} \left|f \left(z\right)\right|^p dA_{\alpha}\left(z\right) < \infty. \end{equation*}

Here $dA_\alpha(z) = (1+\alpha)(1-|z|^2)^\alpha dA(z)$, and $dA\left(z\right) = \frac{1}{\pi}\,dxdy$ is the normalized Lebesgue measure on the unit disk $\mathbb{D}.$ For $\alpha = -1$, we let $A^p_{-1} = H^p$, the standard Hardy space on
the unit disc. 

The Bergman space $A_\alpha^2$ is a closed subspace of
$L^{2}\left(\mathbb{D} , dA_{\alpha}\right)$ and
thus is a Hilbert space with the same inner product inherited from
$L^{2}\left(\mathbb{D} , dA_{\alpha}\right),$ namely
\begin{equation*}
\langle f , g \rangle_{\alpha} = \displaystyle \int_{\mathbb{D}} \, f\,\overline{g}\,dA_{\alpha}.
\end{equation*}

The reproducing kernel of $A_\alpha^2$, denoted by $K^{\left(\alpha\right)}_{z}$,
satisfies 
$f(z)=\langle f, K^{(\alpha)}_{z}\rangle_{\alpha}$, for all $f\in A_\alpha^2$ and $z\in\mathbb{D}$. 
It has the explicit expression  
\begin{equation*}
K^{\left(\alpha\right)}_{z}\left(w\right) = \dfrac{1}{\left(1-\overline{z}\,w\right)^{2 + \alpha}} \quad z, w \in \mathbb{D}.
\end{equation*}

We now define a certain class of functions that we will use repeatedly in
Section 3 to build our testing functions. These functions are formed by taking 
multiples of derivatives of powers of the reproducing kernel for 
the Bergman space $A^2_\alpha$.  
\begin{definition}\label{def:modified_kernel}
For fixed $0 < p \leq \infty$ and $-1 \leq \alpha < \infty$ define 
\[
K_{n,\lambda}(z) = 
  \frac{(1-|\lambda|^2)^{n+1}}{(1 - \overline{\lambda} z)^{n + 1 + \frac{\alpha}{p} + \frac{2}{p}}}.
\]
\end{definition}
It is very important for the rest of the paper that 
\[
\|K_{n,\lambda}\|_{p,\alpha} \asymp C_{n,p,\alpha}
\]
where $C_{n,p,\alpha}$
is a constant depending only on $n$, $p$, and $\alpha$ but not 
on $\lambda$. This follows from the following lemma 
(see \cite[Theorem 1.7]{Zhu_Ap}). 
\begin{lemma}
Let $-1 < \alpha < \infty$ and let $\beta$ be real and suppose 
that $2 + \alpha + \beta > 0$.  Then
\[
\int_{\mathbb{D}} 
   \frac{(1-|w|^2)^\alpha}{|1-z\overline{w}|^{2+\alpha+\beta}} \, dA(w) < 
C (1-|z|^2)^{-\beta},
\]
and
\[
  \int_{\partial \mathbb{D}}
  \frac{1}{|1-z \overline{w}|^{1+\beta}} \, dm(w) <
  C' (1-|z|^2)^{-\beta}
  \]
where $C$ is a constant depending on $\alpha$ and $\beta$ but not 
on $z$, and $C'$ is a constant depending on $\beta$ but not on $z$
and $dm$ is normalized Lebesgue measure on the unit circle. 
\end{lemma}

Also note that the $k^{\text{th}}$ derivative 
\[
K_{n,\lambda}^{(k)}(z) = \left(n+1 + \frac{2+\alpha}{p}\right)_k \, \overline{\lambda}^k 
     \frac{(1-|\lambda|^2)^{n+1}}{(1-\overline{\lambda} z)^{n + 1 + \frac{2+\alpha}{p} +k }} \, ,
\]
where $(a)_b = a(a+1)\cdots(a+b-1)$.

\section{The main interpolation theorem}

We use the following lemma to estimate the norms of our testing functions,
which are found by solving certain systems of linear equations.
The operator norm of an $n \times n$ matrix $A$ is defined by
\begin{equation*}
\|A\| = \displaystyle \sup_{\|x\| = 1} \|Ax\|\, , \, x \in \mathbb{C}^{n},
\end{equation*}
where $\|x\|$ denotes the Euclidean norm of $x$. 
For a basic theory of operator norm of a matrix, see \cite[Chapter 4]{17}. 
\begin{lemma}\label{lemma:inv_det_bound}
Let $A$ be an $n \times n$ matrix with $|\det(A)| \geq D > 0$.  Let 
$\|A\|$ denote the operator norm of $A$.  Then 
\[
\|A^{-1}\| \leq \frac{\|A\|^{n-1}}{D}.
\]
\end{lemma}
\begin{proof}
It suffices to prove that, for every unit vector $x$, we have 
$\|Ax\| \geq D \|A\|^{1-n}$.  Form an orthonormal basis of 
vectors 
$\{x_1, x_2, \cdots, x_n\}$, where $x_1=x$.
Then Hadamard's inequality \cite[Theorem 7.8.1]{17} shows that 
\[
  D\ \leq \prod_{k=1}^n \|A x_k\| \leq \|Ax\| \|A\|^{n-1},
  \]
which gives the desired result.
\end{proof}

We will now prove a series of lemmas culminating in the main interpolation theorem. The theorem states that given an 
integer $N \geq 0$,
there is a constant $C$ such that 
for 
any points $\lambda_0, \ldots, \lambda_N \in \mathbb{D}$, 
and for any integer $J$ such that $0 \leq J \leq N$
there is a function $f \in A^p_\alpha$ with norm at most $1$ such that 
$f^{(j)}(\lambda_j) = 0$ whenever $j \neq J$ and $0 \leq j \leq N$, but that $f^{(J)}(\lambda_J)$ is, up to a factor of at most $C$,
as large as possible. 

The following lemma is used to prove the main interpolation theorem 
in the case where all the $\lambda_j$ are close to the center of the 
unit disc. 
\begin{lemma}\label{lemma:poly_interpolation}
Let $\lambda_0, \ldots, \lambda_N \in \mathbb{D}$ and 
$w_0, \ldots, w_N \in \mathbb{D}$.  There is a constant 
$C$, depending only on $N$, such that there exists a polynomial 
$p$ with $\|p\|_{H^\infty} \leq C$ and $p^{(n)}(\lambda_n) = w_n$ for each $n$ satisfying $0 \leq n \leq N.$ 
\end{lemma}
\begin{proof}
Let 
\(
p(z) = \sum_{n=0}^N a_n z^n.
\)
We must solve the equation
\[
\begin{bmatrix}
1 &  \lambda_0 & \lambda_0^2   & \cdots & \lambda_0^N \\
0 &     1      & 2 \lambda_1   & \cdots & N\lambda_1^{N-1} \\
0 &     0      &  2            & \cdots & N(N-1)\lambda_2^{N-2} \\
\vdots & \vdots & \vdots & \ddots & \vdots\\
0 & 0 & 0 & \cdots & N! 
\end{bmatrix}
\begin{bmatrix} a_0 \\ a_1 \\ a_2 \\ \vdots \\ a_N \end{bmatrix} = 
\begin{bmatrix} w_0 \\ w_1 \\ w_2 \\ \vdots \\ w_N \end{bmatrix}.
\]
But the determinant of the square matrix is $D = \prod_{n=0}^N n!$.  Also, 
the sum of the absolute values of the above matrix is at most 
\[
\sum_{n=0}^N \sum_{j=n}^N \prod_{k=0}^{n-1} (j-k),
\]
so the operator norm is bounded by a constant depending on $N$ 
times the above displayed 
expression.
(This follows since both the operator norm and the $\ell^1$ norm are
norms on the finite dimensional vector space of $(N+1) \times (N+1)$
square matrices \cite[Section 5.6]{17}).   
Thus, by Lemma \ref{lemma:inv_det_bound}, the operator norm of the inverse 
of the square matrix is at most 
$M^{n-1}/D$, where $M$ is the operator norm of the matrix.
  This proves the result. 
\end{proof}

There are three more lemmas before the main interpolation result.  
The first one is used to prove the second, which is used to prove the 
third.  The third one is used in the proof of the main interpolation lemma 
to deal with $\lambda_j$ that are far from the point $\lambda_J$. 
\begin{lemma}
Let $N \geq 0$. 
Suppose $|z_j| \geq \epsilon$ for $1 \leq j \leq M$.  There is a constant $C$ depending 
only on $\epsilon$ and $M$ and $N$ such that there is a function $f$ analytic in the unit disc such that 
$|f(z)| \leq C$ for all $z \in \mathbb{D}$ and $f^{(n)}(0) = \delta_{0 n}$ for $0 \leq n \leq N$ 
and $f^{(n)}(z_j) = 0$ for $0 \leq n \leq N$ and $1 \leq j \leq M$. 
\end{lemma}
\begin{proof}
Let $g(z) = \sum_{n=0}^N (c_n/n!) z^n$.  We will let $f(z) = B(z) g(z)$ where 
$B$ is the Blaschke product with $N$ zeros at each of the $z_j$.  
Let $e_0$ be the vector with a $1$ in position $0$ and all other
entries $0$. Let $c$ be the vector of coefficients of 
$g$ and let $A$ be the $(N+1)\times (N+1)$ matrix defined by 
\[
a_{jk} = \begin{cases} \binom{j}{k} B^{(j-k)}(0) &\text{ for $k \leq j$}\\
0 &\text{ for $k > j$}
\end{cases}
\]
for $0 \leq j, k \leq N$. 
We will define $f$ and $g$ so that $Ac = e_0$.
Note that by the Leibniz rule we have that
\[
  (Bg)^{(j)}(0) = \sum_{k=0}^j \binom{j}{k} B^{(j-k)}(0) g^{(k)}(0),
\]
and thus $Ac = e_0$ implies that $f$ satisfies
$f^{(n)}(0) = \delta_{0 n}$. 
The determinant of $A$ is $B(0)^{N+1}$ which is bounded below by 
$\epsilon^{N(N+1)M}$. The operator norm of $A$ is bounded above by a constant independent of 
$\epsilon$ since all the entries are bounded, which
follows from the facts that
$\sup \{|h^{(n)}(0)| : \|h\|_{H^\infty} \leq 1\} < \infty$ and
$B \in H^\infty.$
Thus, by 
Lemma \ref{lemma:inv_det_bound}, the operator norm of $A^{-1}$ is bounded, and thus the 
$c_j$ are bounded by a constant depending only on 
$\epsilon$ and $N$ and $M$.  
\end{proof}

\begin{lemma}
Let $c > 1$ and $\epsilon > 0$ and an integer $N$ be given.  Suppose that $|z_j| < \epsilon$ for 
$1 \leq j \leq J$ and $|w_j| > c \epsilon$ for $1 \leq j \leq M$.  We can find a constant 
$C$ depending only on $\epsilon$, $c$, $J$, $M$, and $N$ such that there is a function $f$ analytic in the unit disc such that $f^{(n)}(z_j) = \delta_{0 n}$ for $0 \leq n \leq N$ and 
$f^{(n)}(w_j) = 0$ for $0 \leq n \leq N$, and 
$|f(z)| \leq C$ for all $z \in \mathbb{D}$.  
\end{lemma}
\begin{proof}
Let $B$ be the Blaschke product with simple zeros at each $z_j$.  Use the previous lemma 
to find a function $g$ such that 
$g^{(n)}(0) = \delta_{0 n}$ and $g^{(n)}(B(w_j)) = 0$.  The $H^\infty$ norm of 
$g$ is bounded by a constant because 
\[
|B(w_j)| \geq \left( \frac{\epsilon(c-1)}{2}\right)^J.
\]
The function $f(z) = g(B(z))$ now has the required properties.
\end{proof}

Let $\rho$ denote the pseudo-hyperbolic metric, so that 
\[
\rho(z_1,z_2) = \frac{|z_1-z_2|}{|1-\overline{z_1}z_2|}.
\]
The function $\rho$ defines a conformally invariant metric on the unit 
disc (see \cite{D_Ap}).   
\begin{lemma}\label{lemma:aux_interpolation}
  Let $c > 1$ and $\epsilon$ be given.
  Let $J, M \geq 1$ and $N \geq 0$ be integers.
There is a constant
$C$ depending only on $\epsilon$, $c$, $J$, $M$, and $N$ such that if
 $\rho(z_j,z_1) < \epsilon$ for 
  $1 \leq j \leq J$ and $\rho(w_j,z_1) > c \epsilon$ for $1 \leq j \leq M$
then there is a function $f$ analytic in the unit disc such that 
$f^{(n)}(z_j) = \delta_{0 n}$ for $0 \leq n \leq N$ and 
$f^{(n)}(w_j) = 0$ for $0 \leq n \leq N$, and 
$|f(z)| \leq C$ for all $z \in \mathbb{D}$.  
\end{lemma}
\begin{proof}
Let $\varphi$ be a conformal map of the unit disc sending $z_1$ to $0$.  Let $g$ be the function 
found in the previous lemma such that 
$g^{(n)}(\varphi(z_j)) = \delta_{0 n}$ for $0 \leq n \leq N$ and
$1 \leq j \leq J$ and 
$g^{(n)}(\varphi(w_j)) = 0$ for $0 \leq n \leq N$ and $1 \leq j \leq M$, and 
$|g(z)| \leq C$ for all $z \in \mathbb{D}$.  Then 
$f(z) = g(\varphi(z))$ has the required properties, since the pseudo-hyperbolic metric 
is conformally invariant.
\end{proof}

We now come to the main interpolation theorem.  The case where all the 
$\lambda_j$ are small is not difficult.  The next case we consider in the 
proof is the case where all the $\lambda_j$ are close to each other in the 
pseudo-hyperbolic metric.  It is in this case that we use linear 
combinations of the functions 
from Definition \ref{def:modified_kernel} to build the interpolating 
function.  This involves showing that the solution to a certain system 
of linear equations cannot be too large.  The basic idea is to choose the 
functions from Definition \ref{def:modified_kernel} in such a way 
as to make sure the determinant of the matrix of the system is not too small. 
This is done by making the diagonal entries of the matrix much larger 
than the off diagonal entries. 

The last case considered in the proof is if the $\lambda_j$ are not 
necessarily close to $\lambda_J$.  We solve this case by first solving 
a related interpolation problem where all the $\lambda_j$ are close to 
$\lambda_J$, then modifying the solution by using 
Lemma \ref{lemma:aux_interpolation}.

We recall the main interpolation theorem for the sake of the reader. 
\iflabelexists{testfunction}{\setcounter{thrm}{\ref{testfunction}}}{\setcounter{thrm}{3}}
\addtocounter{thrm}{-1}
\begin{thrm} \testfunctiontext \end{thrm}

\begin{proof}
\newcommand{\mv}[1]{\widetilde{#1}}

Fix an $R > 0$ to be determined later. 
If $|\lambda_J| \leq R$ then we may use 
Lemma \ref{lemma:poly_interpolation} to find the function $f$ in the 
statement of the theorem.  

For any integer $m$ let $\mv{m} = m + 1 + \frac{\alpha}{p} + \frac{2}{p}.$  
Let $m_n$ be a sequence of integers increasing so rapidly that for each
$k$ we have 
\begin{equation}\label{eq:mcondition2}
\frac{1}{2} (\mv{m_k})_k\mv{m_k}^{1/2-k} > 1 + 2 \sum_{\substack{0 \leq j \leq N \\j \neq k}} (\mv{m_j})_k \mv{m_j}^{1/2-j}.
\end{equation}
To show that there is such a sequence,  
note that the above inequality will hold (assuming the integers are increasing) 
if $m_0 > N$ and 
\[
  \frac{1}{2} \mv{m_k}^{1/2} > 1 + 2 \sum_{\substack{0 \leq j \leq N \\j \neq k}} 2^k
        \mv{m_j}^k \mv{m_j}^{1/2-j}
\]
which is equivalent to
\[
\mv{m_k}^{1/2} > 2 + 4\sum_{\substack{0 \leq j \leq N \\j \neq k}} 2^k \mv{m_j}^{1/2+k-j}.
\]
But this will certainly be the case if 
\[
\mv{m_k}^{1/2} > 2^{N+2}(N-k)+2 + 4\sum_{0 \leq j < k} 2^k \mv{m_j}^{1/2+k-j}.
\]
This will occur if 
\[
m_k^{1/2} > 2^{N+2}(N + 1) + 2^{k+2} k \mv{m_{k-1}}^{1/2+k}.
\]
So we may choose $m_0 = N+1$ and 
$m_k = \left(2^{N+2}(N+1) + 2^{k+2} k \mv{m_{k-1}}^{1/2+k}\right)^2$. 

If $\rho(z,w) < r$ then
$|z-w| < 4r(1-|z|^2) < 8r(1-|z|)$ 
for $r < \sqrt{1/2}$ (see \cite[p.~41]{D_Ap}). 
Thus if $\rho(z,w) < r$ then 
\[
(1-8r)(1-|z|) < 1-|w| < (1+8r)(1-|z|).
\]
Also 
\[
|z| - 8r(1-|z|) < |w| < |z| + 8r(1-|z|).
\]
Add $1$ to both sides of the above inequality to see that 
\[
(1+|z|)(1-8r) < 1+|z| - 8r(1-|z|) < 1 + |w| < 1 + |z| + 8r(1-|z|) 
       < (1 + |z|)(1+8r).
\]
Thus, for any given $n$, by choosing $r$ small enough we can require that 
\begin{equation}\label{eq:ratio}
2^{-1/2} < \left( \frac{1-|z|^2}{1-|w|^2} \right)^{n} < 2^{1/2}.
\end{equation}

Now consider the disc $D = \{w: \rho(z,w) < r\}$ for fixed $z$.
Note that this is in fact a Euclidean disc (see \cite[p.~40]{D_Ap}). 
We will find bounds for $|1-\overline{w}z| = |1-w\overline{z}|$.  
Without loss of 
generality we may assume $z$ is real and positive. Let $w_r$ be the point in 
$D$ closest to $1$, and $w_l$ be the point in $D$ farthest from 
$1$.  Then $w_r$ and $w_l$ are real. 
If we consider the set $\overline{z}D = \{w\overline{z}: w \in D\}$
we see that it is a 
disc as well.  Furthermore, the point closest to $1$ in
$\overline{z}D$ is no 
closer than $w_r^2$ and the point farthest from $1$ is no farther than 
$w_l^2$.  Thus 
$$\displaystyle \inf_{w \in D} 1 - |w|^2 < |1 - \overline{z} w| 
<
\sup_{w \in D} 1 - |w|^2 .$$ This fact, together with
inequality \eqref{eq:ratio}, shows that 
for given $n$, by choosing $r$ small enough we may require that 
\[
2^{-1/2} < \left( \frac{1-|z|^2}{|1-\overline{w}z|} \right)^{n} < 2^{1/2}.
\]

Thus, there is a radius $0 < R < 1$ and a number $\epsilon$ such that if 
$|z| > R$ and $\rho(z,w) < \epsilon$ then 
$2^{-1/2} < |z|^N, |w|^N < 1$ and 
\[
  2^{-1/2} < \left( \frac{1-|z|^2}{1-|w|^2}
      \right)^{m_N + N + 2 + \tfrac{2+\alpha}{p}} < 2^{1/2}
\]
and
\[
  2^{-1/2} < \left(\frac{|1-\overline{w}z|}{1-|z|^2}
     \right)^{m_N + N + 2 + \tfrac{2+\alpha}{p}}
  < 2^{1/2}.
\]

For the rest of the proof we fix $r$ and $\epsilon$. 
Suppose now that $|\lambda_J| > R$ and that 
$\rho(\lambda_J, \lambda_k) < \epsilon$ for $0 \leq k \leq N$. 
Let $f(z) = \sum_{k=0}^N b_k \mv{m_k}^{1/2-k} K_{m_k, \lambda}$ for some 
constants $b_k$. 
We will try to solve the system of equations
\[
\sum_{k=0}^N b_k \mv{m_k}^{1/2-k} K_{m_k, \lambda_k}^{(j)}(\lambda_j) = 
  \delta_{jJ} \frac{1}{(1-|\lambda_j|^2)^{\frac{2+\alpha}{p} + j}}
\]
for $0 \leq j \leq N$. 
Now we may multiply equation $j$ by 
$(1-|\lambda_j|^2)^{\frac{2+\alpha}{p} + j}$ to see that we 
need to solve a system of the form $A b = e_J$. 
Here $A$ is given by
\[
a_{jk} = 
(\mv{m_k})_j \mv{m_k}^{1/2-k} \frac{(1-|\lambda_j|^2)^{\frac{2+\alpha}{p} + j} 
               (1-|\lambda_k|^2)^{m_k+1} \, \overline{\lambda_j}^j}
{(1-\overline{\lambda_k}\lambda_j)^{m_k + 1 + \frac{\alpha}{p} + \frac{2}{p} + j}}
\]
and $e_J$ is the vector with all $0$ entries except for a $1$ at position $J$. 
Now consider the matrix
\[
\widehat{A} = 
\begin{bmatrix}
\mv{m_0}^{1/2} & \mv{m_1}^{-1/2} & \mv{m_2}^{-3/2} & \cdots & \mv{m_N}^{1/2-N} \\
(\mv{m_0})_1 \mv{m_0}^{1/2} & (\mv{m_1})_1 \mv{m_1}^{-1/2} & (\mv{m_2})_1 \mv{m_2}^{-3/2} & \cdots & (\mv{m_N})_1 \mv{m_N}^{1/2-N} \\
(\mv{m_0})_2 \mv{m_0}^{1/2} & (\mv{m_1})_2 \mv{m_1}^{-1/2} & (\mv{m_2})_2 \mv{m_2}^{-3/2} & \cdots & (\mv{m_N})_2 \mv{m_N}^{1/2-N} \\
\vdots & \vdots & \vdots &\ddots & \vdots \\
(\mv{m_0})_N \mv{m_0}^{1/2} & (\mv{m_1})_N \mv{m_1}^{-1/2} & (\mv{m_2})_N \mv{m_2}^{-3/2} & \cdots & (\mv{m_N})_N \mv{m_N}^{1/2-N} \\
\end{bmatrix}
\]
Then each entry in the diagonal of $A$ differs in absolute value 
by at most a factor of between $1/2$ and 
$2$ from the corresponding 
entry in $\widehat{A}$.  Also, each off diagonal entry in $A$ can be at most twice as large 
as the corresponding entry in $\widehat{A}$.  
The conditions in the inequalities \eqref{eq:mcondition2}
show that the matrix $A$ is diagonally dominant.  
In fact, for each $j$, we have 
\[
1 + \sum_{k \neq j} a_{jk} < a_{jj}.
\]
By Ger\v{s}gorin's theorem \cite[Theorem 6.1.1]{17}, 
all eigenvalues of $A$ are larger than $1$ in absolute value.
Thus the determinant of $A$ is at least $1$ in absolute value. 
This implies that 
the matrix $A$ is invertible and Lemma \ref{lemma:inv_det_bound}
implies that the norm of $A^{-1}$ is bounded by a constant, since 
each entry of $A$ is bounded by a constant. (All of the constants depend 
upon the choice of the sequence $m_k$).   
Therefore, each entry in $b$ is bounded by a constant and 
$f$ has norm that is bounded by a constant.

For the general case where $|\lambda_J| > R$,  let $\epsilon' = \epsilon/(N+1)$.  
By the pigeonhole principle, there must be a number $L$ such that 
$1 \leq L \leq N+1$ and such that the set 
$\{\lambda_k : L\epsilon' < \rho(\lambda_J,\lambda_k) < (L+1)\epsilon'\}$ is 
empty. 

Define the numbers 
\[
z_k = \begin{cases} \lambda_k &\text{ if $\rho(\lambda_J,\lambda_k) < 
         L \epsilon'$}\\
                                 \lambda_J &\text{ if $\rho(\lambda_J,\lambda_k) > (L+1) \epsilon'$}
\end{cases}
\]
Let $\{w_0, \dotsc, w_{N'}\}$ be the set $\{\lambda_0, \dotsc, \lambda_N\}$ with the numbers 
$\{z_1, \ldots, z_N\}$ removed.  Construct a function as above for the 
sequence $\{z_0, \dotsc, z_N\}$; call it $g$.  Construct a function $h$ 
for the sets $\{z_0, \dotsc, z_N\}$ and  $\{w_0, \dotsc, w_{N'}\}$ as in 
Lemma \ref{lemma:aux_interpolation}.
The function $g(z) h(z)$ now has the required properties.
The assertion about the uniform convergence follows since 
$K_{n, \lambda}(z) \rightarrow 0$ as $|\lambda| \rightarrow 1$, uniformly for 
$z$ in compact subsets of $\mathbb{D}$, and the fact that the constants 
$b_k$ and the supremum norm of $h(z)$ are uniformly bounded independently 
of the $\lambda_j$. 
\end{proof}

The significance of the value of $f^{(j)}(\lambda_J)$ is shown by 
the following lemma, found in Lemma 2.1 of \cite{14}, which shows that it is (up to a constant) as large as possible for a function of bounded $A^p_\alpha$ norm. 
\begin{lemma}\label{lemma:growthbound}
Let $0 < p < \infty$ and $-1 \leq \alpha < \infty$, let $n$ be a 
nonnegative integer and let $z \in \mathbb{D}$.  Then 
\[
\sup \left\{ |f^{(n)}(z)|: \|f\|_{p,\alpha} \leq 1 \right\} 
\asymp 
\frac{1}{(1-|z|^2)^{\frac{2 + \alpha}{p} + n}}.
\]
\end{lemma}

\section{Order boundedness and compactness of $T^n_{u; \varphi}$}
Let $T^n_{u; \varphi} = \displaystyle \sum_{k=0}^{n} D^{k}_{u_{k}, \varphi_{k}}$,
as in the introduction.
We say that an operator $T: A^p_\alpha \rightarrow A^q_\beta$
is order bounded
if there is a function $h \in L^q(\mathbb{D}, dA_\beta)$ such that
$|Tf(z)| \leq h(z)$ for all $f \in A^p_\alpha$ such that
$\|f\|_{p,\alpha} < 1$ and for all $z \in \mathbb{D}$. 

The notion of order boundedness is closely related to boundedness.
In fact, it can be easily checked that any order bounded operator $T:A^{p}_{\alpha} \rightarrow A^{q}_{\beta}$ is bounded. Here $p,q \in (0,\infty], \alpha \in [-1,\infty),$ and $\beta \in (-1,\infty).$
The order boundedness of composition operators was first studied by Hunziker and Jarchow in \cite{18}. For more recent work on order bounded composition and weighted composition operators, we refer to \cite{19,20,21,22,23,24,25}.
The main interpolation result, Theorem \ref{testfunction},
is closely related to theorems about the order boundedness and compactness of sums of weighted derivatives of 
composition operators.  

We begin with the following lemma that combines Propositions 2.1 and 2.2 in \cite{14}.
\begin{lemma}\label{lemma:orderbound}
Let $(p,q) \in (0,\infty)$ and $\alpha \in [-1,\infty)$ and 
$\beta \in (-1, \infty)$. Let $n$ be a nonnegative integer. Then the operator $f(z) \mapsto u(z) f^{(n)}(\varphi(z))$ is order bounded from $A^p_\alpha$ to $A^q_\beta$ if and only if 
\[
\int_{\mathbb{D}} \frac{|u(z)|^q}{(1-|\varphi(z)|^2)^{(2+\alpha)\frac{q}{p} + nq}} 
  \, dA_\beta(z) < \infty.
\]
\end{lemma}

We are now ready to prove Theorem \ref{OBExtension}.   
We restate it for the reader's convenience.
\iflabelexists{OBExtension}{\setcounter{thrm}{\ref{OBExtension}}}{\setcounter{thrm}{1}}
\addtocounter{thrm}{-1}
\begin{thrm}\OBExtensiontext\end{thrm}

\begin{proof}%
Due to Lemma \ref{lemma:orderbound} and the fact that a sum of order bounded operators is order bounded, it is sufficient to prove that (a) implies (c). 

Assume (a) holds. Let $T = T^n_{u; \varphi}.$ Then there is some nonnegative function 
$h \in L^q (\mathbb{D}, A_{\beta})$ such that $|Tf(z)| \leq h(z)$ for all 
$f$ of norm at most one in $A^p_\alpha$, and for all $z \in \mathbb{D}$. 

Let $z \in \mathbb{D}$. 
Let $0 \leq J \leq n$ be given and let $f$ be the corresponding
function from Theorem \ref{testfunction} with $N=n$ and 
$\lambda_k = \varphi_k(z)$.  
Then 
\[
|T f(z)| = \left|\frac{u(z)}{(1-|\varphi_J(z)|^2)^{\frac{2 + \alpha}{p} + J}} \right| 
\leq Ch(z) 
\]
for some constant $C$ not depending on $z$ or $J$.  This implies that 
\[
\|T f\|_{q,\beta} = 
\int_{\mathbb{D}} \frac{|u_J(z)|^q}{(1 - |\varphi_J(z)|^2)^{(2+\alpha)\frac{q}{p} + Jq}} dA_\beta (z) \leq \|h\|_{L^q (\mathbb{D}, A_{\beta})}^q < \infty.
\]
\end{proof}
We now discuss compactness of the operators $T^n_{u; \varphi}.$ We first prove a lemma about the compactness of a summand of such an operator. 
\begin{lemma}\label{compact_summand}
  Let $p \in (0,\infty)$ and $\alpha \in [-1,\infty)$.
  Let $u \in H(\mathbb{D})$ and $\varphi$ be an analytic self map of $\mathbb{D}.$  
Then the following are equivalent
\begin{enumerate}
\renewcommand{\theenumi}{\alph{enumi}} 
\item
The
operator $T: f(z) \mapsto u(z) f^{(n)}(\varphi(z))$ is 
compact from $A^p_\alpha$ to $H^\infty$.
\item For every bounded 
sequence $\{f_k\}$ in $A^p_\alpha$ converging to $0$ uniformly on compact subsets of $\mathbb{D}$, the sequence $\|Tf_k\|_\infty$ approaches $0$ as $k \rightarrow \infty$. 
\item
The function
$u \in H^\infty$ and 
\[
\lim_{|\varphi(z)| \rightarrow 1} 
   \frac{|u(z)|}{(1-|\varphi(z)|^2)^{\frac{2+\alpha}{p} + n}} = 0.
\]
\end{enumerate}
\end{lemma}
\begin{proof}
The equivalence of (a) and (b) 
follows from \cite[Lemma 3.7]{26}. To prove (c) implies (b), we use some of the methods from the proof of \cite[Theorem 3.1]{14}. First, suppose that (c) holds. There is thus a number $0 < r < 1$ such that if $|\varphi(z)| > r$ then 
\[
\frac{|u(z)|}{(1-|\varphi(z)|^2)^{\frac{2+\alpha}{p} + n}} < \epsilon.
\]
Thus for $|\varphi(z)| > r$ we have by Lemma \ref{lemma:growthbound} that 
\[
|Tf(z)| = |u(z) f^{(n)}(\varphi(z))| 
\leq
C \|f\|_{p,\alpha} \frac{|u(z)|}{(1-|\varphi(z)|^2)^{\frac{2 + \alpha}{p} + n}}
< C \|f\|_{p,\alpha} \epsilon.
\]
Now if $f_k$ is a sequence of analytic functions with 
$\|f_k\|_{p,\alpha} < A$ and $f_k \rightarrow 0$ uniformly on compact 
subsets of $\mathbb{D},$ then for $|w|\leq r,$ we have that 
$|f_k^{(n)}(w)| < \epsilon$ for sufficiently large $k$.  Thus if 
$|\varphi(z)| \leq r$ we have for sufficiently large $k$ that 
\[
|Tf_k(z)| \leq \|u\|_\infty \epsilon.
\]
Thus $\|Tf_k\| < C' \epsilon$ for sufficiently large $k$, so 
$\|Tf_k\|_\infty \rightarrow 0$.  

Suppose now that (a) and (b) hold.  
Using $f(z) = z^n$ as a test function and 
the fact that $T$ is bounded shows that 
$u \in H^\infty$.  
To prove the rest of (c) it suffices to show that 
\[
   \frac{|u(z_k)|}{(1-|\varphi(z_k)|^2)^{\frac{2+\alpha}{p} + n}}
     \rightarrow 0
\]
for any sequence $z_k$ in $\mathbb{D}$ such that 
$|\varphi(z_k)| \rightarrow 1$.  Let 
\(
f_k = K_{0,\varphi(z_k)} 
\) (see Definition \ref{def:modified_kernel}) 
Then 
\[
\|T f_k\|_\infty \geq |T f_k(\varphi(z_k))| \geq 
\frac{C |u(z_k)| |\varphi(z_k)|^n}{(1-|\varphi(z_k)|^2)^{\frac{2 + \alpha}{p} + n}}.
\]
Now let $k \rightarrow \infty$ to complete the proof. 
\end{proof}

We now prove Theorem \ref{CompactExtension}, which we restate for convenience.
Again, Theorem \ref{testfunction} plays a crucial role.
\iflabelexists{CompactExtension}{\setcounter{thrm}{\ref{CompactExtension}}}{\setcounter{thrm}{2}}
\addtocounter{thrm}{-1}
\begin{thrm} \CompactExtensiontext \end{thrm}
\begin{proof}%
  As before, the equivalence of (a) and (b) follows from
  \cite[Lemma 3.7]{26}. The equivalence of (c) and (d) follows from Lemma \ref{compact_summand}, and the fact that (d) implies (a) is clear.  
To finish the proof we will prove (a) implies (c).    
Assume (a) holds. Using $1$ as a test function implies that $u_0 \in H^\infty$. Using $z$ as a test function and the fact that 
$u_0 \in H^\infty$  implies that $u_1 \in H^\infty$. In general, an induction argument and using $z^k$ as test function shows that 
$u_k \in H^\infty$ for $0 \leq k \leq n$.  The proof of the 
rest of the theorem is similar to Lemma \ref{compact_summand}.  
Let $0 \leq J \leq N$. Let $z_k$ be any sequence in $\mathbb{D}$ such that $|\varphi_J(z_k)| \rightarrow 1$. We let $f_k$ be the function from Theorem \ref{testfunction} 
with $\lambda_j = \varphi_j(z_k)$ and with the given value of $J$.  
Then 
\[
\|T f_k\|_\infty \geq 
|T f_k(\varphi_j(z_k))| \geq 
\frac{C |u(z_k)| |\varphi(z_k)|^J}{(1-|\varphi_J(z_k)|^2)^{\frac{2 + \alpha}{p} + J}}.
\]
Letting $k \rightarrow \infty$ gives the result.  
\end{proof}

\providecommand{\bysame}{\leavevmode\hbox to3em{\hrulefill}\thinspace}
\providecommand{\MR}{\relax\ifhmode\unskip\space\fi MR }
\providecommand{\MRhref}[2]{%
  \href{http://www.ams.org/mathscinet-getitem?mr=#1}{#2}
}
\providecommand{\href}[2]{#2}

\end{document}

